\documentclass{baustms}

\citesort

\newcommand{\dist}{\operatorname{dist}}
\newcommand{\N}{\mathbb{N}}
\newcommand{\R}{\mathbb{R}}
\newcommand{\e}{e}
\newcommand{\D}{\,\mathrm{d}}

\theoremstyle{cupthm}
\newtheorem{thm}{Theorem}[section]

\newtheorem{lemma}[thm]{Lemma}
\theoremstyle{cupdefn}

\theoremstyle{cuprem}
\newtheorem{rem}[thm]{Remark}
\numberwithin{equation}{section}

\begin{document}
\runningtitle{Off-diagonal estimates for the Ornstein--Uhlenbeck semigroup}
\title{$L^p$-$L^q$ off-diagonal estimates for the Ornstein--Uhlenbeck semigroup: some positive and negative results}
\cauthor
\author[1]{Alex Amenta}
\address[1]{Delft Institute of Applied Mathematics, Delft University of Technology, P.O. Box
5031, 2628 CD Delft, The Netherlands\email{amenta@fastmail.fm}}
\author[2]{Jonas Teuwen}
\address[2]{Division of Radiation Oncology, Netherlands Cancer Institute/Antoni van Leeuwenhoek, Plesmanlaan 121,
1066 CX Amsterdam, The Netherlands \\ \\ Department of Imaging Physics, Optics Research Group, Delft University of Technology\email{jonasteuwen@gmail.com}}

\authorheadline{A. Amenta and J. Teuwen}

\support{The first author acknowledges financial support from the Australian Research Council Discovery Grant DP120103692 and the ANR project ``Harmonic analysis at its boundaries'' ANR-12-BS01-0013. The second author acknowledges partial financial support from the Netherlands Organisation for Scientific Research (NWO) by the NWO-VICI grant 639.033.604.}

\begin{abstract}
We investigate $L^p(\gamma)$--$L^q(\gamma)$ off-diagonal estimates for the Ornstein--Uhlenbeck semigroup $(e^{tL})_{t > 0}$.
For sufficiently large $t$ (quantified in terms of $p$ and $q$) these estimates hold in an unrestricted sense, while for sufficiently small $t$ they fail when restricted to maximal admissible balls and sufficiently small annuli.
Our counterexample uses Mehler kernel estimates.
\end{abstract}


\maketitle

\section{Introduction}
Consider the Gaussian measure
\begin{equation}
	\label{eq:Gaussian-measure}
	d\gamma(x) := \pi^{-n/2} e^{-|x|^2} \, dx
\end{equation}
on the Euclidean space $\R^n$, where $n \geq 1$.
Naturally associated with this measure space is the Ornstein--Uhlenbeck operator
\begin{equation*}
	L := \frac 12 \Delta - \langle x, \nabla \rangle = -\frac 12 \nabla^* \nabla,
\end{equation*}
where $\nabla^*$ is the adjoint of the gradient operator $\nabla$ with respect to the Gaussian measure.
This operator generates a heat semigroup $(\e^{tL})_{t > 0}$ on $L^2(\gamma) = L^2(\R^n,\gamma)$, called the Ornstein--Uhlenbeck semigroup, with an explicit kernel: for all $u \in L^2(\gamma)$ and all $x \in \R^n$ we have
\begin{equation*}
	\e^{tL} u(x) = \int_{\R^n} M_t(x,y) u(y) \, d\gamma(y),
\end{equation*}
where 
\begin{equation}\label{eq:mehler-kernel}
	M_t(x,y) = \frac1{(1 - \e^{-2t})^{n/2}} \exp\biggl(-\e^{-t}\frac{|x-y|^2}{1 - \e^{-2t}} \biggr) \exp\biggl(2 \e^{-t}\frac{\langle x,y \rangle}{1 + \e^{-t}}\biggr)
\end{equation}
is the Mehler kernel.
If we equip $\R^n$ with the Euclidean distance and the Gaussian measure, and if we consider operators associated with the Ornstein--Uhlenbeck operator, we find ourselves within the realm of \emph{Gaussian harmonic analysis}: here, the Ornstein--Uhlenbeck operator takes the place of the Laplace operator $\Delta$.\footnote{The multiplicative factor $1/2$, which is not present in the usual definition of the Laplacian, arises naturally from the probabilistic interpretation of the Ornstein--Uhlenbeck operator.}
For a deeper introduction to Gaussian harmonic analysis see the review of Sj\"ogren \cite{pS97} and the introduction of \cite{jT15}.

In this article we investigate whether the Ornstein--Uhlenbeck semigroup satisfies \emph{$L^p(\gamma)$--$L^q(\gamma)$ off-diagonal estimates}: that is, estimates of (or similar to) the form
\begin{equation}
	\label{eq:off-diagonal-1}
  \biggl(\int_{F} |e^{tL} \mathbf{1}_E f|^q \D\gamma \biggr)^{1/q}
  \lesssim t^{-\theta} \exp\Bigl(-c \frac{\dist(E, F)^2}{t} \Bigr)
  \biggl(\int_E |f|^p \D\gamma \biggr)^{1/p},
\end{equation}
for some parameters $c > 0$ and $\theta \geq 0$, where $1 \leq p < q \leq \infty$, $f \in L^p(\gamma)$, and for some class of \emph{testing sets} $E,F \subset X$.
Often such estimates hold whenever $E$ and $F$ are Borel, but in applications we generally only need $E$ to be a ball and $F$ to be an annulus associated with $E$.
Such estimates serve as a replacement for pointwise kernel estimates in the harmonic analysis of operators whose heat semigroups have rough kernels, or no kernels at all, most notably in the solution to the Kato square root problem \cite{AHLMT02} (see also \cite{AKM06}).
Even though the Ornstein--Uhlenbeck semigroup has a smooth kernel, it would be useful to show that it satisfies some form of off-diagonal estimates, as this would suggest potential generalisation to perturbations of the Ornstein--Uhlenbeck operator, whose heat semigroups need not have nice kernels.

Various notions of off-diagonal estimates, including \eqref{eq:off-diagonal-1}, have been considered by Auscher and Martell \cite{AM07}.
However, they only consider doubling metric measure spaces, ruling out the non-doubling Gaussian measure.
Mauceri and Meda \cite{MM07} observed that $\gamma$ is doubling when restricted to \emph{admissible balls}, in the sense that $\gamma(B(x,2r)) \lesssim \gamma(B(x,r))$ when $r \leq \min(1, |x|^{-1})$.
Therefore it is reasonable to expect that the Ornstein--Uhlenbeck semigroup may satisfy some form of $L^p(\gamma)$--$L^q(\gamma)$ off-diagonal estimates if we restrict the testing sets $E,F$ to admissible balls and sufficiently small  annuli.

Here we demonstrate both the success and failure of off-diagonal estimates of the form \eqref{eq:off-diagonal-1}, as a first step in the search for the `right' off-diagonal estimates.
First we give a simple positive result (Theorem \ref{thm:posresult}): for $p \in (1,2)$, and for $t$ sufficiently large (depending on $p$), \eqref{eq:off-diagonal-1} is satisfied for all Borel $E,F \subset \R^n$.
This is proven by interpolating between $L^2(\gamma)$--$L^2(\gamma)$ Davies--Gaffney-type estimates and Nelson's $L^p(\gamma)$--$L^2(\gamma)$ hypercontractivity.
We follow with a negative result (Theorem \ref{thm:mainneg}): for $1 \leq p < q < \infty$ and for $t$ sufficiently small (again depending on $p$ and $q$), \eqref{eq:off-diagonal-1} fails when $E$ is a `maximal' admissible ball $B(c_B,|c_B|^{-1})$ and when $F$ is a sufficiently small annulus $C_k(B)$, in the sense that the implicit constant in \eqref{eq:off-diagonal-1} must blow up exponentially in $|c_B|$.
This is shown by direct estimates of the Mehler kernel.

\subsection*{Notation}
Throughout the article we will work in finite dimension $n \geq 1$.
We will write $L^p(\gamma) = L^p(\R^n,\gamma)$.
Every ball $B \subset \R^n$ is of the form
\begin{equation*}
	B = B(c_B,r_B) = \{x \in \R^n : |x-c_B| < r_B\}
\end{equation*}
for some unique centre $c_B \in \R^n$ and radius $r_B > 0$.
For each ball $B$ and each scalar $\lambda > 0$ we define the expansion $\lambda B = \lambda B(c_B,r_B) := B(c_B,\lambda r_B)$, and we define annuli $(C_k(B))_{k \in \N}$ by
\begin{equation*}
	C_k(B) := \begin{cases} 2B & k = 0, \\ 2^{k+1}B \setminus 2^k B & k \geq 1. \end{cases}
\end{equation*}
For two sets $E,F \subset \R^n$ we write
\begin{equation*}
	\dist(E,F) := \inf\{|x-y| : x \in E, y \in F\}.
\end{equation*}
For two non-negative numbers $A$ and $B$, we write $A \lesssim_{a_1,a_2,\ldots} B$ to mean that $A \leq CB$, where $C$ is a positive constant depending on the quantities $a_1,a_2,\ldots$.
This constant will generally change from line to line. 

\section{A positive result}
The Ornstein--Uhlenbeck semigroup satisfies the following `Davies--Gaffney-type' $L^2(\gamma)$--$L^2(\gamma)$ off-diagonal estimates.
These appear in \cite[Example 6.1]{vNP15}, where they are attributed to Alan McIntosh.
\begin{thm}[McIntosh]\label{thm:mcintosh}
	There exists a constant $C > 0$ such that for all Borel subsets $E,F$ of $\R^n$ and all $u \in L^2(\gamma)$, 
	\begin{equation*}
	\|\mathbf{1}_F \e^{tL} (\mathbf{1}_E u)\|_{L^2(\gamma)}  \leq C \frac{t}{\dist(E,F)} \exp\Bigl(-\frac{\dist(E,F)^2}{2t} \Bigr) \|\mathbf{1}_E u\|_{L^2(\gamma)}.
	\end{equation*}
\end{thm}
Furthermore, Nelson \cite{eN66} established the following hypercontractive behaviour of the semigroup.\footnote{This is done only for $n=1$ in this reference, and a full proof for general $n$ is given in Nelson's seminal 1973 paper \cite{eN73}. These papers won him the 1995 Steele prize.} 
\begin{thm}[Nelson]\label{thm:nelson}
	Let $t > 0$ and $p \in (1 + e^{-2t},2]$.
	Then $\e^{tL}$ is a contraction from $L^p(\gamma)$ to $L^2(\gamma)$.
\end{thm}

Note that $p > 1 + e^{-2t}$ if and only if $t > \frac{1}{2} \log \frac{1}{p-1}$.
Thus the hypercontractive behaviour of the Ornstein--Uhlenbeck semigroup is much more delicate than that of the usual heat semigroup $\e^{t\Delta}$ on $\R^n$, which is a contraction from $L^p(\R^n)$ into $L^q(\R^n)$ for all $1 \leq p \leq q \leq \infty$ and all $t > 0$.

As indicated in the proof of \cite[Proposition 3.2]{pA07}, one can interpolate between Theorems \ref{thm:mcintosh} and \ref{thm:nelson} to deduce certain $L^p(\gamma)$-$L^2(\gamma)$ off-diagonal estimates for the Ornstein--Uhlenbeck semigroup.

\begin{thm}\label{thm:posresult}
	Suppose that $E,F$ are Borel subsets of $\R^n$.
	Let $t > 0$ and $p \in (1 + e^{-2t},2]$.
	Then for all $u \in L^p(\gamma)$,
	\begin{equation*}
		\|\mathbf{1}_F \e^{tL} (\mathbf{1}_E u)\|_{L^2(\gamma)} \leq \biggl( \frac{Ct}{\dist(E,F)} \exp\Bigl(-\frac{\dist(E,F)^2}{2t}\Bigr) \biggr)^{1 - \delta(p,t)} \|\mathbf{1}_E u\|_{L^p(\gamma)},
	\end{equation*}
	where $C$ is the constant from Theorem \ref{thm:mcintosh} and where
	\begin{equation*}
		\delta(p,t) := \frac{\frac12 - \frac1p}{\frac12 - \frac{1}{1 + e^{-2t}}} \in [0,1).
	\end{equation*}
\end{thm}

\begin{proof}
	Write
	\begin{equation*}
		C_M := \frac{Ct}{\dist(E,F)} \exp\Bigl( \frac{\dist(E,F)^2}{2t} \Bigr).
	\end{equation*}
	Theorem \ref{thm:mcintosh} says that
	\begin{equation*}
		\| e^{tL} \|_{L^2(\gamma,E) \to L^2(\gamma,F)} \leq C_M.
	\end{equation*}
	For all $p_0 \in (1 + e^{-2t},p)$ we have
	\begin{equation*}
		\|e^{tL} \|_{L^{p_0}(\gamma,E) \to L^2(\gamma,F)} \leq \|e^{tL}\|_{L^{p_0}(\gamma) \to L^2(\gamma)} \leq 1
	\end{equation*}
	by Theorem \ref{thm:nelson}.
	Therefore by the Riesz--Thorin theorem we get
	\begin{equation*}
		\|e^{tL} \|_{L^p(\gamma,E) \to L^p(\gamma,F)} \leq C_M^{\theta(p_0)},
	\end{equation*}
	where $p^{-1} = (1-\theta(p_0))/p_0 + \theta(p_0)/2$, or equivalently
	\begin{equation*}
		\theta(p_0) = \frac{\frac1p - \frac{1}{p_0}}{\frac12 - \frac{1}{p_0}} = 1 - \frac{\frac12 - \frac1p}{\frac12 - \frac{1}{p_0}}.
	\end{equation*}
	Taking the limit as $p_0 \to 1 + e^{-2t}$ gives
	\begin{equation*}
		\|e^{tL}\|_{L^p(\gamma,E) \to L^p(\gamma,F)} \leq C_M^{1 - \delta(p,t)}
	\end{equation*}
	and completes the proof.
\end{proof}

\begin{rem}
	For $1 < p < q < \infty$, a $L^p(\gamma)$--$L^q(\gamma)$ version of Theorem \ref{thm:posresult} could be proven by first establishing $L^q(\gamma)$--$L^q(\gamma)$ off-diagonal estimates---which may be obtained by interpolating between boundedness on $L^q(\gamma)$ and the Davies--Gaffney type estimates---and then arguing by the $L^p(\gamma)$--$L^q(\gamma)$ version of Nelson's theorem.
\end{rem}

This positive result does not rule out the possibility of some \emph{restricted} $L^p(\gamma)$--$L^2(\gamma)$ off-diagonal estimates for $p \leq 1 + e^{-2t}$.
In the next section we show one way in which these can fail.

\section{Lower bounds and negative results}
\label{sec:fail}

In this section we show that the $L^p(\gamma)$--$L^q(\gamma)$ off-diagonal estimates of \eqref{eq:off-diagonal-1} are not satisfied for admissible balls and small annuli when $t$ is sufficiently small (depending on $p$ and $q$).
More precisely, we show that \eqref{eq:off-diagonal-1} fails when $E$ is a maximal admissible ball $B$, i.e.\ a ball for which $r_B = \min(1,|c_B|^{-1})$, and $F$ is an annulus $C_k(B)$ with $k$ sufficiently small.
These sets typically appear in applications of off-diagonal estimates.

\begin{thm}\label{thm:mainneg}
	Suppose that $1 \leq p < q < \infty$, and that
	\begin{equation}\label{eqn:tpg-condn}
		\frac{2}{e^t + 1} > 1 - \biggl( \frac{1}{p} - \frac{1}{q} \biggr),
	\end{equation}
	or equivalently that
	\begin{equation*}
		t < \log\Biggl( \frac{1 + (\frac{1}{p} - \frac{1}{q})}{1 - (\frac{1}{p} - \frac{1}{q})} \Biggr).
	\end{equation*}
	Then the off-diagonal estimates \eqref{eq:off-diagonal-1} do not hold for the class of testing sets
	\begin{equation*}
		\{(E,F) : E = B(c_B,|c_B|^{-1}), \, F = C_k(B), \, 2^k \leq |c_B|\}.
	\end{equation*}
\end{thm}
Note that $\frac{1}{p} - \frac{1}{q} \in (0,1)$, so we always obtain some range of $t$ for which the off-diagonal estimates \eqref{eq:off-diagonal-1} fail.

Let us compare Theorems \ref{thm:mainneg} and \ref{thm:posresult}.
Having fixed $p \in (1,2)$, we get failure of $L^p(\gamma)$--$L^2(\gamma)$ off-diagonal estimates for maximal admissible balls and small annuli for $e^{tL}$ when $t < \log\left( \frac{1 + (\frac{1}{p} - \frac{1}{2})}{1 - (\frac{1}{p} - \frac{1}{2})} \right)$, and when $t > \frac{1}{2} \log \frac{1}{p-1}$ the off-diagonal estimates hold for all Borel sets.
We do not know what happens for the remaining values of $t$.

To prove Theorem \ref{thm:mainneg} we rely on the following lower bound.
\begin{lemma}\label{lem:main}
	Suppose $k \geq 1$ is a natural number, $1 < q < \infty$, and let $B$ be a maximal admissible ball with $|c_B| \geq 2^k$.
	Then
	\begin{equation*}
		\biggl( \int_{C_k(B)} | (\e^{tL} \mathbf{1}_B)(y) |^q \, \D\gamma(y) \biggr)^{1/q}
		\gtrsim_{k,n,t} |c_B|^{-n(1+\frac{1}{q})} \exp\biggl( |c_B|^2\biggl( \frac{2}{e^t + 1} - 1 - \frac{1}{q}\biggr) \biggr).
	\end{equation*}
\end{lemma}

\begin{proof}[Proof of Lemma \ref{lem:main}]
Suppose $x \in B$ and $y \in C_j(B)$.
We argue by computing a lower bound for the Mehler kernel $M_t(x,y)$ as given in \eqref{eq:mehler-kernel}.

First we focus on the factor involving the inner product $\langle x, y \rangle$, where $x = (x_1, x_2, \ldots, x_n)$ and $y = (y_1, y_2, \ldots, y_n)$. 
By symmetry we may assume that $c_B = |c_B| e_1$.
Using $r_B = |c_B|^{-1}$, we get that
\begin{equation*}
 x_1 y_1 \geq (|c_B| - r_B)(|c_B| - 2^{k+1} r_B) \geq |c_B|^2 + O(1),
\end{equation*}
where we use the big-O notation $O(1)$ to mean that $x_1 y_1 - |c_B|^2$ is bounded as $|c_B| \to \infty$. 
If $n \geq 2$, then by using $x_i y_i = O(1)$ for $i \geq 2$ we get that
\begin{equation*}
	\langle x, y \rangle \geq |c_B|^2 + O(1);
\end{equation*}
evidently this estimate remains true when $n=1$.

Using the Mehler kernel representation of $e^{tL}$, for all $y \in C_k(B)$ we thus have
\begin{align*}
	e^{tL} \mathbf{1}_B(y)
	&\gtrsim_{n,t} \int_B \exp\biggl(-e^{-t} \frac{|x-y|^2}{1-e^{-2t}} \biggr) \exp\biggl( \frac{2|c_B|^2}{e^t + 1} \biggr) \D\gamma(x).
\end{align*}
Since $|x-y| < 2^{k+1} r_B \leq 2$, using $r_B = |c_B|^{-1} \leq 2^{-k}$, this gives
\begin{align}
	e^{tL} \mathbf{1}_B (y)
	&\gtrsim_{n,t} \exp\biggl( \frac{2|c_B|^2}{e^t + 1} \biggr) \gamma(B) \nonumber\\
	&\gtrsim_n |c_B|^{-n} \exp\biggl( \frac{2|c_B|^2}{e^t + 1}  - (|c_B| + |c_B|^{-1})^2 \biggr) \nonumber\\
	&\simeq |c_B|^{-n} \exp\biggl( |c_B|^2 \biggl( \frac{2}{e^t + 1} - 1 \biggr) \biggr) \label{eqn:e-est}
\end{align}
using a straightforward estimate on $\gamma(B)$.
Next, we estimate
\begin{align*}
	\gamma(C_k(B)) &\gtrsim_n |C_k(B)| e^{-(|c_B| + 2^{k+1} r_B)^2} \\
	&\simeq_{n} 2^{kn} r_B^n \exp\biggl( -(|c_B|^2 + 2^{k+2}|c_B|r_B + 2^{k+2}r_B^2) \biggr) \\
	&\simeq_{k,n} |c_B|^{-n} \e^{-|c_B|^2}.
\end{align*}
Combining this with \eqref{eqn:e-est} gives
\begin{align*}
	\biggl( \int_{C_k(B)} |(e^{tL} \mathbf{1}_B(y))|^q \D\gamma(y) \biggr)^{1/q}
	&\gtrsim_{n,t} |c_B|^{-n} \exp\biggl( |c_B|^2 \biggl( \frac{2}{e^t + 1} - 1 \biggr) \biggr)  \gamma(C_k(B))^{1/q} \\
	&\gtrsim_{k,n} |c_B|^{-n(1+\frac{1}{q})} \exp\biggl( |c_B|^2\biggl( \frac{2}{e^t + 1} - 1 - \frac{1}{q}\biggr) \biggr),
\end{align*}
as claimed.
\end{proof}

\begin{proof}[Proof of Theorem \ref{thm:mainneg}]
	We argue by contradiction.
	Suppose that $e^{tL}$ satisfies the $L^p(\gamma)$--$L^q(\gamma)$ off-diagonal estimates \eqref{eq:off-diagonal-1} for some $\theta \geq 0$, and for $(E,F)$ as stated.
	Fix a natural number $k \geq 1$ and let $B$ be a maximal admissible ball with $|c_B| > 2^k$.
	Lemma \ref{lem:main} and the off-diagonal estimates for $E = B$, $F = C_k(B)$, and $f = \mathbf{1}_B$ then imply
	\begin{align*}
		|c_B|^{-n(1+\frac{1}{q})} &\exp\biggl( |c_B|^2\biggl( \frac{2}{e^t + 1} - 1 - \frac{1}{q}\biggr)  \biggr) \\
		&\lesssim_{k,n,t,\theta} \exp\biggl(-c \frac{(2^{k+1}-1)^2 r_B^2}{t} \biggr) \gamma(B)^{1/p} \\
		&\simeq \gamma(B)^{1/p}
	\end{align*}
	for some $c > 0$.
	Since
	\begin{equation*}
		\gamma(B)^{1/p} \lesssim_n |B|^{1/p} e^{-\frac{1}{p}(|c_B| - r_B)^2} \simeq_n |c_B|^{-n/p} \exp\biggl(-\frac{|c_B|^2}{p}\biggr),
	\end{equation*}
	this implies
	\begin{align*}
		\exp\biggl( |c_B|^2\biggl( \frac{2}{e^t + 1} - 1 + \frac{1}{p} - \frac{1}{q} \biggr)  \biggr) \lesssim_{k,n,t,\theta} |c_B|^{n(1 - (\frac{1}{p} - \frac{1}{q}))}.
	\end{align*}
	The left hand side grows exponentially in $|c_B|$ when \eqref{eqn:tpg-condn} is satisfied.
	However, the right hand side only grows polynomially in $|c_B|$.
	Thus we have a contradiction.
\end{proof}

\begin{rem}
	By the same argument we can prove failure of $L^p(\gamma)$--$L^q(\gamma)$ off-diagonal estimates for the derivatives $(L^m e^{tL})_{m \in \N}$ of the Ornstein--Uhlenbeck semigroup, with the same conditions on $(p,q,t)$ and the same class of testing sets $(E,F)$.
	This relies on an identification of the kernel of $L^m e^{tL}$, which has been done by the second author \cite{jT16}.
\end{rem}

In this article we only considered off-diagonal estimates with respect to the Gaussian measure $\gamma$.
In future work it would be very interesting to consider appropriate weighted measures, following in particular \cite{BBG12} and \cite{BBGM12}, in which (among many other things) it is shown that estimates of the form $\| e^{tL} f\|_{L^2(\gamma)} \lesssim \|f V_t\|_{L^1(\gamma)}$ hold, where $V_t$ is a certain weight depending on $t$.
Thus the Ornstein--Uhlenbeck semigroup does satisfy a form of `ultracontractivity', but with the caveat that one must keep track of $t$-dependent weights.
It seems that this has not yet been explored in the context of Gaussian harmonic analysis.

\acks The authors thank Mikko Kemppainen, Jan van Neerven, and Pierre Portal for valuable discussions and encouragement on this topic.
We also thank an anonymous referee for their suggested simplification of the proof of Lemma \ref{lem:main}.

\bibliographystyle{srtnumbered}
\bibliography{OU}

\end{document}